\documentclass[a4paper,12pt]{article}

\usepackage{color}
\usepackage[pdftex]{graphicx}
\usepackage{amsmath,amssymb,amsthm}
\usepackage{wrapfig}
\usepackage{ifpdf}

\newtheorem{theorem}{Theorem}[section]
\newtheorem{corr}{Corollary}[section]
\newtheorem{lemma}{Lemma}[section]
\newtheorem{prop}{Proposition}[section]

\newtheorem{df}{Definition}[section]

\newcommand{\conv}{\mathop{\rm conv}\nolimits}
\newcommand{\dist}{\mathop{\rm dist}\nolimits}
\newcommand{\rk}{\mathop{\rm rank}\nolimits}

\newcommand{\rp}{\mathop{\rm dim^E_2}\nolimits}
\newcommand{\sr}{\mathop{\rm  dim^S_2}\nolimits}
\newcommand{\ur}{\mathop{\rm  dim^J_2}\nolimits}
\newcommand{\aff}{\mathop{\rm aff}\nolimits}

\newcommand{\kk}{\mathop{\rm K}\nolimits}

\newcommand{\cg}{\mathop{\rm CG}\nolimits}

\title {Graphs and spherical two--distance sets}

\author {Oleg R. Musin\thanks{This research is partially supported by the NSF grant DMS--1400876.}}

\begin{document}
\date{}
\maketitle

\begin{abstract} 
Every graph $G$ can be embedded in a Euclidean space as a two--distance set. The Euclidean representation number of $G$ is the smallest dimension in which $G$ is representable by such an embedding. We consider spherical and J--spherical representation numbers of $G$ and give exact formulas for these numbers using  multiplicities of polynomials that are defined by  the Caley--Menger determinant. One of the main results of the paper are explicit formulas for the representation numbers of the join of graphs which are obtained from W. Kuperberg's  type theorem for two--distance sets.  
\end {abstract}

Throughout this paper we will consider only simple graphs, ${\mathbb R}^{d}$ will denote the $d$--dimensional Euclidean space,  ${\Bbb S}^{n}$ will denote the $n$--dimensional unit sphere in ${\mathbb R}^{n+1}$,  and $\dist(x,y):=||x-y||$ will denote the Euclidean distance in  ${\Bbb R}^d$. For a set $X\subset {\mathbb R}^{d}$ we shall denote the affine hull (or affine span) by $\aff(X)$, $\rk(X):=\dim{\aff(X)}$ and  $\conv(X)$  will denote the convex hull of $X$. We will denote the cardinality of a finite set $X$ by $|X|$ .  

\section{Introduction}

Representations (embeddings) of a graph $G$  into a metric space, in particular into  ${\Bbb R}^d$,  is a classical discrete geometry problem (see \cite[Ch. 6,19] {DezaLa} and \cite[Ch. 15,19] {Deza2}). The dimension of  $G$ is the smallest $d$ for which it can be embedded in  ${\Bbb R}^d$ as a unit--distance graph \cite{bezdek2013}.  In this paper we consider  the smallest $d$ for which $G$ can be embedded as a two--distance set. 

Let $G$ be a graph on $n$ vertices. Consider a {\em Euclidean representation of $G$} in ${\Bbb R}^d$ as a two--distance set. In other words, there are two positive real numbers $a$ and $b$ with $b\ge a>0$ and an embedding $f$ of the vertex set of $G$ into  ${\Bbb R}^d$ such that 
$$ \dist(f(u),f(v)):=\left\{
\begin{array}{l}
a \; \mbox{ if } uv \mbox{ is an edge of } G\\
b \;  \mbox{ otherwise}
\end{array} 
\right.
$$
After Roy \cite{Roy2010}, the smallest $d$ such that $G$ is representable in  ${\Bbb R}^d$ we will call the {\em Euclidean representation number} of $G$ and denote it $\rp(G)$.

Einhorn and Schoenberg \cite{ES} showed that  $\rp(G)$ can be found explicitly in terms of the multiplicity  $\mu(G)$ of the root $\tau_1$ of the discriminating polynomial (see Section 2). 

\medskip

\noindent{\bf Theorem \ref{thER}} {\em Let $G$ be a graph with $n$ vertices. Then 
$$  \rp(G)=n-\mu(G)-1.$$
}

In Section 3 we consider representations of $G$ as spherical two--distance sets.  Let $f$ be a Euclidean representation of  $G$ in ${\Bbb R}^d$ with the minimum distance $a=1$.   We say that $f$ is {\em spherical} if the image $f(G)$ lies on a $(d-1)$--sphere in ${\Bbb R}^d$.  We denote by $\sr(G)$ the smallest $d$ such that $G$ is spherically representable in  ${\Bbb R}^d$.  

If $d\le n-2$, then  $f$ is uniquely defined up to isometry (see Section 2). Therefore, if $f$ is spherical, then the circumradius of $f(G)$ is also uniquely defined.  We denote it by ${\mathcal R}(G)$. If $f$ is not spherical or $\mu(G)=0$, then we put ${\mathcal R}(G)=\infty$  (Definition \ref{DefR}).

\medskip

\noindent{\bf Theorem \ref{cor21}.} {\em Let $G$ be a graph with $n$ vertices. Then 
$$  \sr(G)=\left\{
\begin{array}{l}
\rp(G), \quad {\mathcal R}(G)<\infty;\\
n-1, \qquad \: \, {\mathcal R}(G)=\infty. 
\end{array} 
\right.
$$
}


 Nozaki and Shinohara \cite{NS2012} also give necessary and sufficient conditions of a Euclidean representation of $G$ to be spherical. However, their conditions are more bulky. Namely, they used Roy's theorem (see  \cite[Theorem 2.4]{NS2012}) and they showed that among five types of conditions only three of them yields sphericity \cite[Theorem 3.7]{NS2012}. 

Nozaki and Shinohara also considered strongly regular graphs. For instance, they proved the following interesting fact: {\em a graph $G$ with $n$ vertices is strongly regular if and only if} $\sr(G)+\sr(\bar G)+1=n$  \cite[Theorem 4.5]{NS2012}.

\medskip 

Theorem \ref{th23} states that  ${\mathcal R}(G)\ge1/\sqrt{2}$. In Section 4 we  consider the extreme case  ${\mathcal R}(G)=1/\sqrt{2}$.    Let $f$ be a spherical representation of a graph $G$  in ${\Bbb R}^d$  as a two--distance set.  We say that $f$ is a {J--spherical representation of $G$} if the image $f(G)$ lies in a sphere ${\mathbb S}^{d-1}$ of radius  $1/\sqrt{2}$ and  the first  (minimum) distance $a=1$.  

To prove the existence of J--spherical representations is not very easy.  Corollary \ref{corr41}  states that for any graph $G\ne K_n$ there is a unique (up to isometry) J--spherical representation. 
Then for a J--spherical representation $f:G\to{\Bbb R}^d$  the dimension $d$ and second distance $b$ are uniquely defined, we denote these $d$ and $b$  by $\ur(G)$ and   $\beta_*(G)$ respectevely.

\medskip

\noindent{\bf Theorem \ref{Jrepth}.} {\em Let $G\ne K_n$  be a graph on $n$ vertices. Then 
$$  \ur(G)=\left\{
\begin{array}{l}
\rp(G), \quad {\mathcal R}(G)=1/\sqrt{2};\\
n-1, \qquad \: \,  {\mathcal R}(G)>1/\sqrt{2}. 
\end{array} 
\right.
$$
}



In Section 5 we consider W. Kuperberg's theorem on sets $S$ in  ${\Bbb S}^{n-1}$ with $n+2\le|S|\le2n$ and the minimum distance between points of $S$ at least $\sqrt{2}$ \cite{kuperberg2007}. Theorem \ref{KTD} shows that $S$ is the join of its subsets $S_i$. If $S$ is a two--distance set, then $S$ is a J--spherical representation.  

Using results of Section 5, in Section 6 we give explicit formulas for representation numbers in the case when $G$ is the graph join: $G=G_1+\ldots+G_m$. In particular, these formulas can be applied for the complete multipartite graph  $K_{n_1\ldots n_m}$.

\medskip 

\noindent {\bf Theorem \ref{thJG}.}  {\em Let  $G_1,\ldots,G_m$ be a finite collection of graphs with $n_1,\ldots,n_m$ vertices respectively, let $G:=G_1+\ldots+G_m$ and  $n:=n_1+\ldots+n_m$.  Suppose   
$$\beta_*(G_1)=\ldots=\beta_*(G_{k})<\beta_*(G_{k+1})\le\ldots\le\beta_*(G_m).$$ 
 Then 
$$\ur(G)=\ur(G_1)+\ldots+\ur(G_k)+n_{k+1}+\ldots+n_m, $$
$$\sr(G)=\ur(G), \qquad \rp(G)=\min(\ur(G),n-2).$$
}

\medskip

\noindent{\bf Corollary \ref{corJ}.}  {\em  Let $G$ be the complete multipartite graph $K_{n_1\ldots n_m}$.
Suppose 
$$n_1=\ldots=n_{k} > n_{k+1}\ge\ldots\ge n_m$$ 
and let $n:=n_1+\ldots+n_m.$  Then
 \begin{enumerate}
 \item  $\rp(G)=\min(n-k,n-2);$
 \item $\sr(G)=\ur(G)=n-k$.
  \end{enumerate}
}


Note that Statement 1 in Corollary \ref{corJ} was first proved by Roy \cite[Theorem 1]{Roy2010}. 

\medskip 

In Section 7 we consider seven open problems on representations of graphs. 



\section{Euclidean representations of graphs} 
In this section we consider Euclidean representations of graphs as two--distance sets. 

A complete graph $K_n$ represents the vertices of a regular $(n -1)$--simplex.  In fact, this is a representation of $K_n$ as a  one--distance set. Then  $\rp(\kk_n)=n-1$ and 
$$
\rp(G)\le n-1
$$
for any graph $G$ with $n$ vertices.

Thus we have a correspondence between graphs and two-distance sets. Let  $S$ be a two--distance set in ${\Bbb R}^d$ with distances $a$ and $b\ge a$. Denote by $\Gamma(S)$ a graph with  $S$  as the set vertices and edges $[pq]$, $p,q\in S$, such that $\dist(p,q)=a$. Then $S$ is a Euclidean representation of $G=\Gamma(S)$.

Let $S$ be a two--distance set of cardinality $n$ in ${\Bbb R}^d$. Then, see \cite{bannai1983,blokhuis1984}, we have 
$$
n\le\frac{(d+1)(d+2)}{2}.  \eqno (2.1)
 $$ 
(Lison{\v{e}}k \cite{lisonek1997} shows that the upper  bound (2.1) is tight for $d=8$.) This bound  implies the following lower bound  
$$
\rp(G)\ge\frac{\sqrt{8n+1}-3}{2}. 
$$

 Let $G$ be a graph with $n$ vertices. 
 Einhorn and Schoenberg \cite{ES} considered Euclidean representations of graphs. They  proved that
 
 \medskip
 
  {\em $\rp(G)=n-1$ if and only if $G$ is a disjoint union of cliques.}
  
  \medskip

\noindent Moreover,  they  have shown that 

\medskip

{\em If $\rp(G)\le n-2$, then a Euclidean representation of $G$ in ${\Bbb R}^d$, where $d:=\rp(G)$, is uniquely defined up to isometry.}

\medskip

Let $S=\{p_1,\ldots,p_n\}$ be a two-distance set with distances $a=1$ and $b>1$. Let  $d_{ij}:=\dist(p_i,p_j)$.
Consider the Cayley--Menger determinant 

$$
C_S:=
 \begin{vmatrix}
 0 & 1 & 1 & ... & 1\\
 1 & 0 & d_{12}^2 & ... & d_{1n}^2\\
  1  & d_{21}^2 & 0 & ... & d_{2n}^2\\
\hdotsfor{5}\\
\hdotsfor{5}\\
  1 & d_{n1}^2 & d_{n2}^2 & ... & 0
  \end{vmatrix} 
\eqno (2.2)
 $$
  Since for $i\ne j$, $d_{ij}=1$ or $b$,  $C_S$ is a polynomial in $t=b^2$. Denote this polynomial by $C_G(t)$.

 Actually, in \cite{ES} instead of $C_G$ the discriminating polynomial $D(t)$ is considered. This polynomial can be defined through the Gram determinant. Since, see \cite[Lemma 9.7.3.3]{Berger},  
 $$C_G(t)=(-1)^nD(t)$$  
 these two polynomials are the same up to the sign and therefore have the same roots.

\begin{df} Let $G$ be a graph with $n$ vertices.    Let $\tau_1=\tau_1(G)$ be the smallest root of $C_G(t)$, i.e. $C_G(\tau_1)=0,$ such that $\tau_1> 1$.  By $\mu(G)$ we denote the multiplicity of the root $\tau_1(G)$ of $C_G$. If all roots $t_*\le 1$, then we put $\tau_1(G)=\infty$ and $\mu(G)=0$. 
\end{df}

 Einhorn and Schoenberg  proved that if $S$ is embedded exactly in ${\Bbb R}^d$, then $\tau_1$ is a root of $C_G(t)$ of exact multiplicity $n-d-1$ \cite[Lemma 6]{ES}. Equivalently, we have the following theorem:

\begin{theorem} \label{thER} Let $G$ be a graph with $n$ vertices. Then 
$$\rp(G)=n-\mu(G)-1.$$
\end{theorem}

Roy \cite{Roy2010} found that $\rp(G)$ depends on certain eigenvalues of graphs. Actually,  these  dimensions are closely related with the multiplicity of the smallest (or second smallest) eigenvalue of the adjacency matrix $A(G).$ 
  
  In \cite{ES, NS2012,Roy2010}  two Euclidean representation numbers $\rp(G)$ and $\rp(\bar G)$ are considered , where $\bar G$ is the graph complement of $G$. These numbers can be different. For instance, let $G$ be the disjoint union of $m$ edges. Then $\rp(G)=2m-1$. On the other hand, $\bar G$ is the complete multipartie graph $K_{2,\ldots,2}$. It follows from \cite[Theorem 2]{ES} or \cite[Theorem 1]{Roy2010} (see also \cite{alexandrov2016}) that $$\rp(K_{2,\ldots,2})=m.$$ 
Indeed,  $G=K_{2,\ldots,2}$, then $n=2m$ and
  $$C_G(t)=2m\,t^m(2-t)^{m-1}.$$
  Therefore $\tau_1(G)=2$ and $\mu(G)=m-1$. Thus $\rp(K_{2,\ldots,2})=m.$
  
  Note that a minimal Euclidean representation of this graph is a regular $m$--dimensional cross--polytope. In Section 6 we consider a geometric method for complete multipartite graphs. 
  
  There is an obvious relation between polynomials $C_G(t)$ and $C_{\bar G}(t)$. Namely,  $C_{\bar G}(t)$ is the reciprocal polynomial of $C_G(t)$. If $G$ or $\bar G$ is not the complete multipartite graph,  then   $\tau_0(G):=1/\tau_1(\bar G)$ is a root of $C_G(t)$ and there are no more roots in the interval $I:=[\tau_0(G),\tau_1(G)]$. Moreover, a two--distance set $S$ with distances 1 and $\sqrt{t}$ is well--defined only if $t\in I$ \cite{ES}. 
    
  In fact, if $\rp(G)\le n-2$, then a minimal Euclidean representation is unique up to isometry. Indeed, in this case $a=1$ and $b=\sqrt{\tau_1}$, then all distances between vertices in the representation are known.

  Using this approach Einhorn and Schoenberg \cite{ES} enumerated all two-distance sets in dimensions two and three. In other words, they enumerated all graphs $G$ with $\rp(G)=2$ and $\rp(G)=3$.  In \cite{musin2018} we state the same  problem in four dimensions. Recently, Sz\"oll\"osi \cite{Sz2018}   using a computer enumeration of graphs solved this problem.  


%

  
\section{Spherical representations of graphs} 

Let $f$ be a Euclidean representation of a graph $G$ with $n$ vertices in ${\Bbb R}^d$  as a two--distance set. 
We say that $f$ is a {\em spherical representation of $G$} if the image $f(G)$ lies on a $(d-1)$--sphere in ${\Bbb R}^d$.  We will call  the smallest $d$ such that $G$ is spherically representable in  ${\Bbb R}^d$ the {\em spherical representation number} of $G$ and denote it $\sr(G)$. 

Representation numbers $\sr(G)$ and $\rp(G)$ can be different. In Section 6 we show that if $G$ is a bipartite graph $K_{m,n}$ with $m\ne n$, then $$\rp(K_{m,n})=n+m-2<\sr(K_{m,n})=n+m-1.$$

For a graph $G$ on $n$ vertices we obviously have 
$$
\rp(G)\le \sr(G)\le n-1 \eqno (3.1)
$$

Actually, for spherical representation numbers  lower bound (2.1) can be a little bit improved. 
Delsarte, Goethals, and Seidel \cite{delsarte1977} proved that the largest cardinality of spherical two-distance sets in ${\Bbb R}^d$ is bounded by $d(d+3)/2$. (This upper bound is known to be tight for $d=2,6,22$.)  That yields 
$$
\sr(G)\ge\frac{\sqrt{8n+9}-3}{2}. 
$$

This bound has been improved for some dimensions.  Namely, in  \cite{musin2009} we proved that 
$$n\le \frac{d(d+1)}{2} \eqno (3.2)$$ 
for $6<d<22$ and  $23<d<40$. This inequality was  extended for almost all $d\le 93$ by Barg \& Yu \cite{barg2013} and  for $d\le 417$ by Yu \cite{Yu2016}. Recently, Glazyrin \& Yu  \cite{Glazyrin2016} proved (3.2) for all $d\ge 7$ with possible exceptions for some $d=(2k+1)^2-3$, $k\in {\mathbb N}$.  

\medskip

Let $S=\{p_1,\ldots,p_n\}$ be a  set in ${\Bbb R}^{n-1}$. As above, $d_{ij}:=\dist(p_i,p_j)$. Let 
$$
M_S:=
 \begin{vmatrix}
  0 & d_{12}^2 & ... & d_{1n}^2\\
 d_{21}^2 & 0 & ... & d_{2n}^2\\
\hdotsfor{4}\\
\hdotsfor{4}\\
 d_{n1}^2 & d_{n2}^2 & ... & \, 0
  \end{vmatrix}
  \eqno (3.3)
 $$
 It is well known \cite[Proposition 9.7.3.7]{Berger}, that if the points in $S$ form a simplex of dimension  $(n-1)$, then the radius $R$ of the sphere circumscribed around this simplex is given by 
 $$
 R^2=-\frac{1}{2}\frac{M_S}{C_S}. \eqno (3.4)
 $$
(Here $C_S$ is defined by (2.2).)

\begin{df} \label{df31}
Let $G$ be a graph with vertices $v_1,\ldots,v_n$. Put $d_{ij}:=1$ if $[v_iv_j]$ is an edge of $G$, otherwise put  $d_{ij}:=b$.   We denote by $C_G(t)$ and $M_G(t)$ the polynomials in $t=b^2$ that are defined by (2.2) and (3.3),   respectively.  Let 
$$
F_G(t):=-\frac{1}{2}\frac{M_G(t)}{C_G(t)}. 
$$
\end{df}


\begin{lemma} \label{L31}
Let $S$ be a spherical representation of a graph $G$ with distances $a$ and $b$, $b\ge a$. Then $S$ lies on a sphere of radius $R=\sqrt{a^2F_G(b^2/a^2)}$.
\end{lemma} 
\begin{proof} If $X=\{x_1,\ldots,x_n\}$ is a set of points in ${\mathbb R}^{n-1}$  in general position, then $\rk(X)=n-1$, $\conv(X)$ is a simplex and (3.4) determines the circumradius $R(X)$ of $\conv(X)$. Clearly,  $R(X)$  is a continues function in $\{x_i\}$. 

We have  that $\rk(S)\le n-1$. If $\rk(S)=n-1$, then (3.4) implies the lemma, otherwise consider a sequence of sets $\{X_k\},$ $k\in{\mathbb N}$,  in ${\mathbb R}^{n-1}$  in general position such that $S$ is a limit set  of this sequence.  Thus, $R(S)$ is the limit of $\{R(X_k)\}$, $k\in{\mathbb N}$. 
\end{proof}

As we noted above, if  $\rk(S)<n-1$ and $a=1$, then a spherical  (and Euclidean) representation of $G$ is uniquely defined up to isometry. However, if $\rk(S)=n-1$, then there are infinitely many non--isometric spherical representations. This is easy to see, let $S$ be the set of vertices of a simplex in which one of edges has length $b\ge 1$ and all other edges are of lengths $a=1$. It can be proved (see the next section) that the range of $R(S)$ is $[1/\sqrt{2},\infty)$.  This fact and Lemma \ref{L31} explain our definition of the circumradius of $G$.

\begin{df} \label{DefR} 
If $G$ is a graph with $\tau_1(G)<\infty$ and $F_G(\tau_1)<\infty$, then denote 
$${\mathcal R}(G):=\sqrt{F_G(\tau_1)}.$$ Otherwise, put ${\mathcal R}(G):=\infty$. 

\end{df} 


\begin{theorem} \label{cor21}  Let $G$ be a graph on $n$ vertices. Then 
$$  \sr(G)=\left\{
\begin{array}{l}
\rp(G) \, \mbox{ if } \; {\mathcal R}(G)<\infty;\\
n-1 \quad \; \, \mbox{ if } \;  {\mathcal R}(G)=\infty. 
\end{array} 
\right.
$$
\end{theorem}


\begin{proof}  Denote by $I_\varepsilon$ a small interval $[\tau_1-\varepsilon,\tau_1+\varepsilon]$  that does not contain any other roots of $C_G$ and $M_G$.  Then for every $t$ in $I_\varepsilon$, $t\ne \tau_1$, the Cayley--Menger determinant (2.2) is non--zero. Therefore, it defines a Euclidean (spherical) representation $f_t$ of $G$ in ${\Bbb R}^{n-1}$. Let $S_t:=\{f_t(v_i)\}$, where $v_i$ are the vertices of $G$.   Lemma \ref{L31} implies that  
$F_G(t)=R^2(t),$ 
where $R(t)$ is the radius of the sphere circumscribed about $S_t$. 

 From (3.1) it follows that  $\rp(G)=n-1$ yields $\sr(G)=n-1$. If  $\rp(G)\le n-2$, then $\mu(G)\ge 1$. Therefore,   for $t=\tau_1$,  Theorem \ref{thER}  implies that $S_t$ is embedded into ${\Bbb R}^{n-\mu-1}$.

Suppose $\sr(G)\le n-2$. Then (3.1) implies that $\rp(G)\le n-2$.  In this case a minimal spherical representation of $G$ is uniquely defined by $\tau_1$ and $S_{\tau_1}$ is a spherical set that lies on a sphere of radius $\rho>0$. Then $R(t)$ and $F_G(t)$ are continuous functions in $t$ that are well defined for all $t$ in $I_\varepsilon$  and $F_G(\tau_1)=\rho^2$.  It is easy to see that the inequlality $F_G(\tau_1)>0$  yields that the multiplicities of $\tau_1$ in $C_G$ and $M_G$ are equal. Thus, we have $\sr(G)=\rp(G)$. 
\end{proof}


 \section{ J--spherical representation of graphs}

In this section we prove that ${\mathcal R}(G)\ge1/\sqrt{2}$ and then we consider  the boundary case ${\mathcal R}(G)=1/\sqrt{2}$. 

For a proof  of the next theorem we need Rankin's theorem.  Rankin \cite{Rankin} proved that \\{\em If $S$ is a set of $d + k$, $k\ge 2$, points in the unit sphere ${\mathbb S}^{d-1}$ in ${\mathbb R}^{d}$, then two of the points in $S$ are at a distance of at most $\sqrt{2}$ from each other.}

\begin{theorem} \label{th23}   ${\mathcal R}(G)\ge1/\sqrt{2}$. 
\end{theorem} 
\begin{proof} Let $G$ be a graph on $n$ vertices. By the definition if $\sr(G)= n-1$, then ${\mathcal R}(G)=\infty$. 

 Let $S$ be  a minimal spherical representation of $G$. If  $\sr(G)\le n-2$, then $S$ lies in a sphere $\Omega$ in ${\mathbb R}^{n-2}$ of radius $R$. By Rankin's theorem  if  $d+2$ points lie
in a sphere of radius $R$ in ${\mathbb R}^{d}$, then a ratio $a/R\le\sqrt{2}$, where $a$ is the minimum distance between these points. Since $a=1$, we have ${\mathcal R}(G)=R\ge1/\sqrt{2}$. 
 \end{proof}

Hence we have a two--distance set $X$ with distances $a=1$ and $b> a$ such that the circumradius of $X$ is $1/\sqrt{2}$.  Actually, we will consider a set $S$ that is similar to $X$  with the scale factor $\sqrt{2}$.  Therefore, $S$ is a two-distance set with the first distance   $a=\sqrt{2}$ that can be inscribed in the unit sphere.

\begin{df} Let $f$ be a spherical representation of a graph $G$  in ${\Bbb R}^d$  as a two distance set. 
We say that $f$ is a {J--spherical representation of $G$} if the image $f(G)$ lies in the unit sphere ${\mathbb S}^{d-1}$ and the first  (minimum) distance $a=\sqrt{2}$.  
\end{df}

The existence of Euclidean and spherical representations  for any graph $G$ is obvious. However, to prove it for J--spherical representations is not very easy. Clearly, if\, $G$ is a complete graph $K_n$, then this representation does not exist. We show that this is just one exceptional case, and for every other $G$ there is a J--spherical representation. 

\medskip

\noindent{\bf Notation}. {\em Let $G$ be a graph on $n$ vertices. \\ 
$I_G:=\left(\sqrt{2},\sqrt{2\tau_1(G)}\right)$.\\  
$S_G(x) :$  a  two--distance set $S$ in ${\mathbb R}^{n-1}$ with distances $a=\sqrt{2}$ and $b=x$ such that
$\Gamma(S)=G$.   (Here, as above, $\Gamma(S)$ is the graph with edges of length $a$.) \\
$\Delta_G(x):=\conv{S_G(x)}$. \\
$\Phi_G(x) :$ the radius of the minimum enclosing ball of $S_G(x)$  in\, ${\mathbb R}^{n-1}$.
}

\begin{lemma} \label{lemma31} If \, $x\in I_G$, then $\rk{S_G(x)}=n-1$. 
\end{lemma}

\begin{proof}
Since the Cayley--Menger determinant  and the volume of a simplex are equal up to a constant and  $C_G(x^2/2)\ne0$ for $x\in I_G$,  we have that $\Delta_G(x)$ is a  simplex in  ${\mathbb R}^{n-1}$ of dimension $n-1$. Thus,  $\rk{S_G(x)}=\dim{\Delta_G(x)}=n-1$. 
\end{proof}

\begin{lemma} \label{lemma32} The function  $\Phi_G(x)$ is increasing on $I_G$. 
\end{lemma}

\begin{proof}
The proof relies on the Kirszbraun theorem (see \cite{Kirsz, AT})\footnote{The author thanks Arseniy Akopyan for the idea of this proof.}: 

\medskip

 {\em Let $X$ be a subset of ${\mathbb R}^d$ and $f:X\to {\mathbb R}^m$ be a Lipschitz function. Then $f$ can be extended to the whole ${\mathbb R}^d$ keeping the Lipschitz constant of the original function.}
 
 \medskip

Let $\sqrt{2}\le y_1<y_2<\sqrt{2\tau_1(G)}$. Then by Lemma \ref{lemma31} $S_G(y_i)=\{v_{i1},\ldots,v_{in}\}$ is the set of vertices of an $(n-1)$--simplex $\Delta_G(y_i)$  that lies in the minimum enclosing ball $B(y_i)$ of radius $\Phi_G(y_i)$. 

Let 
$$
h(v_{2k}):=v_{1k}, \; k=1,\ldots,n. 
$$
Then we have $h:S_G(y_2)\to{\mathbb R}^{n-1}$. It is clear that the Lipschitz constant of $h$ is equal to 1.  By the Kirszbraun theorem   $h$ can be extended to $H:{\mathbb R}^{n-1}\to{\mathbb R}^{n-1}$ with the same  Lipschitz constant.

 Let $c_2$ be the center of $B(y_2)$.  For all $k=1,\ldots, n$ we have 
 $$
 \dist(H(c_2),H(v_{2k}))= \dist(H(c_2),v_{1k})\le \dist(c_2,v_{2k})\le \Phi_G(y_2). 
 $$
 Therefore, $H(c_2)$ is a point in $\Delta_G(y_1)$ such that all distances from $H(c_2)$ to vertices $S_G(y_1)$ does not exceed $\Phi_G(y_2)$. Then $\Phi_G(y_1)\le \Phi_G(y_2)$.  
\end{proof}

 
\begin{lemma} \label{prop1} Let $S$ be a set in \, ${\mathbb R}^{n-1}$ of cardinality $|S|\ge n$. Suppose the minimum distance between points of $S$ is at least $\sqrt{2}$. If $S$ lies in a sphere of radius $R\le1$, then sphere's center $O\in\conv(S)$. 
\end{lemma}      
\begin{proof} Assume the converse. Then $S$ lies in an open hemisphere of radius $R$.    It can be proved (see \cite[Theorem 3]{Mus24} or   \cite[Theorem 5]{BM}) that the assumptions yield $|S|<n$, a contradiction.   
\end{proof}


\begin{theorem} \label{Jrep}  
Let  $G$ be a graph with $n$ vertices. Let $R: \sqrt{(n-1)/n} < R\le 1$. Suppose $G\ne K_n$, then there is a unique $x\in I_G$ such that  $S_G(x)$ lies on a sphere of radius $R$. 
\end{theorem}

\begin{proof}   Let $b_1:=\sqrt{2\tau_1(G)}$. First we prove that there is a solution of the equation $\Phi_G(x)=R$.  Namely, we are going to prove that  $$\Phi_G(\sqrt{2})= \sqrt{(n-1)/n} \le R \le 1\le \Phi_G(b_1).$$

 Indeed, it is clear that  $\Phi_G(\sqrt{2})$ is the circumradius of a regular $(n-1)$--simplex, 
 of side length $\sqrt{2}$.  Then 
 $$
\Phi_G(\sqrt{2})=\frac{n-1}{n}.  
 $$

Now we show that $\Phi_G(b_1)\ge1$. 
 In the case $b_1=\infty$, it is clear that $\Phi_G(x)$  approaches $\infty$ as  $x$ approaches $\infty$.

Let $b_1<\infty.$ Then the Cayley--Menger determinant  vanishes and $S_G(b_1)$ embeds in ${\mathbb R}^{d}$, where $d\le n-2$.  By Theorem \ref{th23}, $\sqrt{2}{\mathcal R}(G)\ge1$. Therefore, if $\Phi_G(x)<1$, then $x<b_1$.  

(Equivalently, we have $n\ge d+2$ points with the minimum distance $\sqrt{2}$ in a ball of radius $\Phi_G(b_1)$. By Rankin's theorem \cite{Rankin} it is possible only if the radius  $\Phi_G(b_1)\ge1$.)

Therefore $\Phi_G(b)=R$ for some $b\in[\sqrt{2},b_1]$. 
 
\medskip

Now we show that  for $x\in[\sqrt{2},b_1]$ a solution of the equation  $\Phi_G(x)=R$ is unique.  By Lemma \ref{lemma32} $\Phi_G(x)$ is increasing whenever $x$ is increasing. However, we did not prove that  $\Phi_G(x)$ is a strictly increasing function. 
Suppose $P(y_1)=R$ and $P(y_2)=R$, where $y_1<y_2$.  Then $\Phi_G(x)$ is a constant on the interval $[y_1,y_2]$.  Lemma \ref{prop1}  yields that for $x\in[y_1,y_2]$ the circumcenter of a simplex  $\Delta_G(x)$ lies in this simplex.

It is well known that if the circumcenter  of a simplex $\Delta$ is an internal point of $\Delta$, then the minimum enclosing sphere is the circumsphere of $\Delta$. Therefore, for this case we have 
$$
\Phi_G(x)=\sqrt{2F_G(t)}, \; t=\frac{x^2}{2}. 
$$
Then  $\Phi^2_G(x)$ is a rational function in $x^2$.  It implies that $\Phi_G(x)$ cannot be a constant in  $[y_1,y_2]$.

\medskip

 Note that the case of an empty graph, i.e. $G=\bar K_{1,\ldots,1}$, is well--defined.  If $R=1$, then  
$$
 b_*=\sqrt{\frac{2n}{n-1}} > \sqrt{2} 
 $$
and $S_G(b(1))$ is the set of vertices of a regular $(n-1)$--simplex of side length $b$. 
(In this case there are no edges of length $a=\sqrt{2}$.) If for $R<1$ we take $b=Rb_*$, then it will be a unique  solution of the equation $\Phi_G(x)=R$.  
\end{proof} 

This theorem for $R=1$ yields the following 

\begin{corr} \label{corr41} For every graph $G\ne K_n$ there is a unique (up to isometry) J--spherical representation. 
\end{corr}

The uniqueness of a J--spherical representation shows that the following definition is correct.

 \begin{df} Let $f:G\to{\Bbb R}^d$ be a J--spherical representation of $G$.   We denote the image $f(G)$  by $W_G$ and the dimension $d$  by $\ur(G)$.  
  Denote the second distance of $W_G$ by $\beta_*(G)$. 
\end{df}

Representation numbers $\ur(G)$ and $\sr(G)$ can be different. For instance, if $G$ is the pentagon, then 
$$
\sr(G)=2<\ur(G)=4. 
$$

Note that $\ur(G)<n-1$ only if $\beta_*(G)=\sqrt{2\,\tau_1(G)}$. Moreover, since the circumradius of $W_G$ is 1, we have to have ${\mathcal R}(G)=1/\sqrt{2}$. That yields the following theorem.

\begin{theorem} \label{Jrepth}  Let $G\ne K_n$  be a graph on $n$ vertices. Then 
$$  \ur(G)=\left\{
\begin{array}{l}
\rp(G), \quad {\mathcal R}(G)=1/\sqrt{2};\\
n-1, \qquad \: \,  {\mathcal R}(G)>1/\sqrt{2}. 
\end{array} 
\right.
$$
\end{theorem}

 Rankin's theorem  and Theorem \ref{Jrepth} yield 

\begin{corr} 
 Let $G$ be a graph on $n$ vertices and  $G\ne K_n$.   Then 
 $$
 \frac{n}{2}\le\ur(G)\le n-1. 
 $$
 If $\ur(G)=n/2$, then $G=K_{2,\ldots,2}$ and a  $J$--spherical  representation of $G$ is a regular cross--polytope. 
 \end{corr}


\section{The join of sets and Kuperberg's theorem} 

\subsection{W. Kuperberg's theorem.} 

 As we noted above, Rankin's theorem states that   if $S$ is a subset of ${\mathbb S}^{d-1}$ with $|S|\ge d+2$, then the minimum distance between points in $S$ is at most $\sqrt{2}$. 
 Wlodzimierz Kuperberg \cite{kuperberg2007} extended Rankin's theorem  and proved that: 
 \begin{theorem} \label{KT}  Let $d$ and $k$ be integers such that $2\le k \le d.$ If $S$ is a
 ($d + k)$--point subset of the unit $d$--ball such that the minimum distance between points is at least $\sqrt{2}$, then: (1) every point of $S$ lies on the boundary of the ball, and (2) ${\Bbb R}^d$ splits into the orthogonal product $\prod_{i=1}^k{L_i}$ of nondegenerate linear subspaces $L_i$ such that for  $S_i:=S\cap L_i$ we have $|S_i|=d_i+1$ and $\rk(S_i)=d_i$  $(i = 1, 2, . . . , k)$, where  $d_i:= \dim{L_i}$. 
 \end{theorem}
 
 In fact, this theorem states that $S$ is join--decomposable. 
 
\begin{df} The join $X*Y$  of two sets $X\subset{\mathbb R}^m$ and $Y\subset{\mathbb R}^n$ is formed in the following  manner.   Embed $X$ in  the $m$--dimensional linear  subspace  of\, ${\mathbb R}^{m+n}$ as 
$$\{(x_1,\ldots,x_m,0,\ldots,0): x=(x_1,\ldots,x_m)\in X\}$$ and embed $Y$  as 
$$\{(0,\ldots,0,y_1,\ldots,y_n): y\in Y\}.$$
\end{df}
 
 Geometrically the  join corresponds to putting the two sets $X$ and $Y$ in orthogonal  linear subspaces of ${\mathbb R}^{m+n}$. Hence Kuperberg's theorem implies that $S=S_1*\ldots*S_k$.

 Actually, Kuperberg's proof of Theorem \ref{KT} yields that  $\conv(S_i)$ contains the center $O$ of the unit $d$--ball. This statement also follows from Lemma \ref{prop1}

Let $\conv(S)$ be a $d$--dimensional simplex, i.e. $\rk(S)=d$.  We have two cases: \\
 (i)  $O$ lies in the interior of $\conv(S)$; \\
 (ii) $O$  lies on the boundary of  $\conv(S)$. 
 
 It is clear, that in Case (i) $S$ is join--indecomposable. Consider Case (ii). Let $S_1$ be a minimal subset of $S$ among such subsets whose convex hull contains $O$. Then \cite[Proposition 6]{kuperberg2007} yields that $S_2:=S\setminus S_1$ lies in the orthogonal complement of $\aff(S_1)$, i.e. $S=S_1*S_2$.

\begin{lemma} \label{prop2} Let $S$ be a subset of \, ${\mathbb S}^{d-1}$ with $|S|\ge d+1$ such that the minimum distance between points of $S$ is at least $\sqrt{2}$. Suppose $O$  lies on the boundary of  $\conv(S)$. Then $S$ is join--decomposable. 
\end{lemma}

This lemma shows that  there are two types of join--indecomposable spherical sets. 

\medskip

\noindent
{\bf Type I:} $S\subset{\mathbb S}^{d-1}$,  $|S|=d+1$,    $\rk(S)=d$ and the center $O$ of\, ${\mathbb S}^{d-1}$   lies in the interior of $\conv(S)$. 

\medskip

\noindent
{\bf Type II:} $S\subset{\mathbb S}^{d-1}$,  $|S|=d$,  $\rk(S)=d-1$ and  $O\notin{\aff(S)}$. 

\medskip

Consider an example, let $S$ consists of three vertices of an isosceles right triangle in the unit circle, for instance, $S=\{p_1.p_2,p_3\}$, $p_1=(1,0),\, p_2=(-1,0)$ and $p_3=(0,1)$. Then $S=S'*S''$, where $S':=\{p_1,p_2\}$ and $S'':=\{p_3\}$. Here $S'$ is of Type 1 and $S''$ is of Type 2. 

Lemma \ref{prop2} says that if $O$ lies in the boundary of $S_i$ then $S_i=S_i'*S''_i$.  It yields the following version  of Kuperberg's theorem. 

\begin{theorem} \label{th32} Let $S$ be a subset of the unit $d$--ball in ${\mathbb R}^{d}$ with the minimum distance between points at least $\sqrt{2}$. Suppose $|S|=d+k$ with $2\le k\le d$. Then $S=S_1*\ldots*S_m$, where $S_i,\, i=1,\ldots,k$ are of Type I and all other $S_i$ are of Type II. 
\end{theorem}

\subsection{The join of spherical two--distance sets}

\begin{df}  We say that a two--distance set $S$ in ${\mathbb R}^d$ is a J--spherical two--distance set (JSTD) if $S$ lies in the unit sphere centered at the origin 0 and $a=\sqrt{2}$.  For this $S$ the second distance $b$ will be denoted $b(S)$. 
\end{df}

 The next two lemmas immediately follow from definitions.  
\begin{lemma} \label{lemma42} Let $S_1$ and $S_2$ be spherical two-distance sets with the same distances $a$ and $b\ge a$. Let $R_i$ denote the circumradius of $S_i$. Then (1) the join $S_1*S_2$ is spherical if $R_1=R_2$ and (2) the join is a two--distance set  only if $R^2_1+R_2^2=a^2$ or $R^2_1+R_2^2=b^2. $
 \end{lemma}

\begin{lemma} \label{lemma43} Let $S_1$ and $S_2$ be JSTD sets with $b(S_1)=b(S_2).$   Then  the join $S_1*S_2$ is a JSTD set. 
 \end{lemma}

  \begin{lemma} \label{lemma44} Suppose  for sets $X_1$ and $X_2$ in  ${\Bbb R}^d$ there is positive $\rho$ such that  $\dist(p_1,p_2)=\rho$ for all points $p_1\in X_1$, $p_2\in X_2$. 
 Then  both $X_i$ are spherical sets and  the affine hulls   $\aff(X_i)$   in ${\Bbb R}^d$  are orthogonal each other. If additionally  $\rk(X_1\cup 0)+\rk(X_2\cup 0)=\rk(X_1\cup X_2\cup0)$, then $X_1\cup X_2=X_1*X_2$, where $0$ denote the origin of\, ${\Bbb R}^d$.
 \end{lemma}

\begin{proof} 1. If $p\in X_1$, then by assumption $X_2$ lies on a sphere $S_\rho(p)$  of radius $\rho$ and centered at $p$. Therefore, $X_2$ belongs to a sphere that is the intersection of all $S_\rho(p)$, where $p\in X_1$.   \\
2. Let $p_1,p_2\in X_1$ and $q_1,q_2\in X_2$. Since in the tetrahedron $p_1p_2q_1q_2$ four sides $p_iq_j$ have the same length $\rho$, the edges $p_1p_2$ and $q_1q_2$ are orthogonal. That implies the orthogonality of  the affine spans  $\aff(X_1)$ and  $\aff(X_2)$  in  ${\Bbb R}^d$. \\ 
3. Let $L_i:=\aff(X_i\cup 0)$. 
Then $\dim{L_i}=\rk(X_i\cup 0)$. By assumption $L_1\cap L_2=0$.  Thus, the orthogonality of $\aff(X_i)$ yields $X_1\cup X_2=X_1*X_2$.  
\end{proof} 

\begin{theorem} \label{thJS2} Let $S_1$ and $S_2$ be JSTD sets in ${\Bbb R}^d$.  Then   $S:=S_1\cup S_2$ is a JSTD set and $S=S_1*S_2$   if and only if (1) $\dist(p_1,p_2)$ are the same for all points $p_1\in S_1$, $p_2\in S_2$;  (2) $\rk(S\cup 0)=\rk(S_1\cup 0)+\rk(S_2\cup 0)$ and (3) $b(S_1)=b(S_2)$.  
 \end{theorem} 
 
 \begin{proof} By Lemma \ref{lemma44}, (1) and (2) imply that $S=S_1*S_2$. Since $R_1=R_2=1$, from Lemma   \ref{lemma42} we have $\dist(p_1,p_2)=\sqrt{2}$. Finally, Lemma  \ref{lemma43} yields that $S$ is JSDT.  
  \end{proof}

\subsection{Kuperberg type theorem for two--distance sets}

  \begin{df} Let $S$ be a two-distance set. We say that $S$ is J--prime if $S$ is indecomposable with respect to the join. 
  \end{df}
  
  It is easy to see that J--prime sets can be defined in another way. 
  
  \begin{prop} Let $S$ be a two--distance set. Let $G=\Gamma(S)$. Then $S$ is J--prime if and only if the graph complement $\bar G$ is connected.
 \end{prop}

  From Theorem \ref{th32} we know that any J--prime set is of Type I or Type II. If $S$ is of Type I in ${\mathbb R}^d$, then $S$ is a JSTD of rank $d$ and cardinality $d+1$. Therefore if we take $G=\Gamma(S)$, then we obtain $S=W_G$.  Note that the inequality $\beta_*(G)<\sqrt{\tau_1(G)}$ implies that $\ur(G)=d$, where $G$ is a graph on $d+1$ vertices.  We proved the following:

   \begin{lemma} \label{lemma45} Let  $S$ be  a J--prime JSTD set of Type I. Then $b(S)=\beta_*(G)<\sqrt{\tau_1(G)}$, where $G:=\Gamma(S)$.  
 \end{lemma}

If $S$ is of Type II in ${\mathbb R}^d$, then $S$ is a JSDT set of cardinality $d$. For instance, if $S=\{p,q\}$ is a two--points set in the unit circle with $\sqrt{2}<b= \dist(p,q)<2$, then $S$  is J-- prime of Type II.  Hence in this case the second distance $b$ is not fixed and lies in some open interval. 

Let $S$ be a JSDT set in  ${\mathbb R}^d$ of cardinality $d+k$, where $2\le k\le d$. For this $S$  Theorem \ref{th32}  states that there are exactly $k$ subsets $S_i$ of Type I. Now if we take $S_1$ of Type I and $S_2$ of Type II then $S_1*S_2$  is a JSDT set. From Lemma \ref{lemma43} follows that $b(S_1)=b(S_2)$. Moreover, for $S_2$ we have an extra constraint: this set lies in a $(d-2)$--sphere of radius $R<1$.

\begin{lemma} \label{lemma46} A JSTD set  $S$ in ${\mathbb R}^d$, $d=|S|-2$, is a J--prime set of Type II  only if $b(S)<\beta_*(G)<\sqrt{\tau_1(G)}$, where $G:=\Gamma(S)$.  
\end{lemma}
\begin{proof} The assumption $b(S)<\beta_*(G)$ is equivalent to $R<1$, where $R$ is  the circumradius of $S$. By Theorem \ref{Jrep}, there is a unique $b$ such that a two--distance set $S$ with $a=\sqrt{2}$ lies in a sphere of radius $R$.  
\end{proof}

  Theorem \ref{th32} implies the following theorem. 
 
 \begin{theorem} \label{KTD} Let $S$, $|S|=d+k$, $k\ge1$, be a two--distance set in the unit sphere  in ${\Bbb R}^{d}$ with the minimum distance  $a=\sqrt{2}$. Then $S=S_1*\ldots*S_m$   such that all subsets  $S_i$ are J--prime and exactly $k$ of them are of Type I.
 \end{theorem}


\section{Representation numbers of the join of graphs}

Recall that the {\em join} $G=G_1+G_2$ of graphs $G_1$ and $G_2$ with disjoint point sets $V_1$ and $V_2$ and edge sets $E_1$ and $E_2$ is the graph union $G_1\cup G_2$ together with the edges joining each point of $V_1$ to each point of $V_2$.  In this section we apply results of Section 5 for the join of graphs.   

The following theorem is a version of Theorem \ref{KTD}. 

 

 \begin{theorem} \label{th52} Let $G$ be a graph with $n$ vertices. Let\, $\ur(G)=n-k\le n-2$. Then   $G=G_1+\ldots+G_m$, where all  $G_i$ are  indecomposable with respect to the join and   
 $$\beta_*(G)=\beta_*(G_1)=\ldots=\beta_*(G_{k})<\beta_*(G_{k+1})\le\ldots\le\beta_*(G_m).$$ 
  \end{theorem}
  
 \begin{proof} Let $S$ be a J--spherical representation of $G$. Then $S$ satisfies the assumptions of Theorem \ref{KTD}. Therefore $S=S_1*\ldots*S_m$. Let $S_1,\ldots,S_k$ be sets of Type I. Thus subgraphs $G_i:=\Gamma(S_i)$ are as required.  
 \end{proof}

\begin{theorem} \label{thJG} Let  $G_1,\ldots,G_q$ be a finite collection of graphs with $n_1,\ldots,n_q$ vertices, respectively. Let $G:=G_1+\ldots+G_q$ and  $n:=n_1+\ldots+n_q$.  Suppose   
$$\beta_*(G_1)=\ldots=\beta_*(G_{p})<\beta_*(G_{p+1})\le\ldots\le\beta_*(G_q).$$ 
 Then 
$$\ur(G)=\ur(G_1)+\ldots+\ur(G_p)+n_{p+1}+\ldots+n_q, $$
$$\sr(G)=\ur(G), \qquad \rp(G)=\min(\ur(G),n-2).$$
\end{theorem}

\begin{proof} By Theorem \ref{th52} there are graphs  $F_1,\ldots,F_m$  indecomposable with respect to the join and such that $G:=F_1+\ldots+F_m$, $k:=k_1+\ldots+k_p$, where $k_i:=n_i-\ur(G_i)$,  and    
$$\beta_*(F_1)=\ldots=\beta_*(F_{k})<\beta_*(F_{k+1})\le\ldots\le\beta_*(F_m).$$

 Let $S_i:=W_{F_i}$, $i=1,\ldots k$. For $i>k$, denote by $S_i$ a sets of Type II with $\Gamma(S_i)=F_i$ and $b(S_i)=\beta_*(F_1)$. Then let $S=S_1*\ldots*S_m$ be a J--spherical representation of $G$. It is clear that  $\rk(S)=n-k$.

If $k\ge2$, then  $\ur(G)\le \rk(S)\le n-2$.  In this case Lemma \ref{lemma42}, Theorem \ref{cor21} and Theorem \ref{Jrepth} yield 
$$
\rp(G)=\sr(G)=\ur(G)=n-k=\ur(G_1)+\ldots+\ur(G_p)+n_{p+1}+\ldots+n_q.
$$ 

Now consider the case  $\ur(G)=n-1$ or, equivalently, $k=1$. Let $H:=F_2+\ldots+F_m$. Note that $\beta_*(F_1)<\beta_*(H)=\beta_*(F_2)$. 

Since $G$ is not a disjoint union of cliques, $\rp(G)\le n-2$. Therefore, a Euclidean representation $f:G=F_1+H \to {\mathbb R}^{n-2}$  is unique. Let $X_1:=f(F_1)$ and $X_2:=f(H)$.   
From Lemma \ref{lemma44} it follows that $X_1$ and $X_2$ are spherical orthogonal sets. Moreover, by Lemma \ref{lemma42} we have $R^2_1+R^2_2=a^2$, where $R_i$ denotes the circumradius of $X_i$. 

First note that $R_1\ne R_2$, otherwise $X$ and $Y$ would be JSTD sets with $\rp(G)=\ur(G)=n-1$.   Hence $f$ would not a spherical representation and $\sr(G)=n-1$. 

Note that $R_1>R_2$. Indeed, it follows from the fact that  $b(X_1)=b(X_2)$, but $\beta_*(F_1)<\beta_*(H)$. Since  $b(X_2)<\beta_*(H)$, we have $\rk(X_2)=v_H-1$, where $v_H$ denotes the number of vertices of $H$. Thus  $\rp(G)=\rk(X_1\cup X_2)=v_1-1+v_H-1=n-2$. 
  \end{proof}

\begin{corr} \label{corJ} Let $G$ be the complete multipartite graph $K_{n_1\ldots n_m}$ and $n:=n_1+\ldots+n_m.$ 
Suppose 
$$n_1=\ldots=n_{k} > n_{k+1}\ge\ldots\ge n_m.$$ 
 Then
$$
\sr(G)=\ur(G)=n-k, \qquad \rp(G)=\min(n-k,n-2)
$$
\end{corr}

\begin{proof} Note that $$K_{n_1\ldots n_m}=\bar K_{n_1}+\ldots+\bar K_{n_m}.$$ Since 
$$
\beta_*(\bar K_n)=\sqrt{\frac{2n}{n-1}},
$$
our assumption is equivalent to 
$$\beta_*(\bar K_{n_1})=\ldots=\beta_*(\bar K_{n_{k}}) < \beta_*(\bar K_{n_{k+1}})\le\ldots\le \beta_*(\bar K_{n_m}).$$ 
Thus, this corollary follows from Theorem \ref{thJG} and the obvious fact that the empty graph  $\bar K_{\ell}$ is indecomposable with respect to the join, i.e. $\ur(\bar K_{\ell})=\ell-1$. 
\end{proof}


\section{Concluding remarks and open problems} 

First we consider open problems that are directly related to this paper. 

\subsection{Range of the circumradius ${\mathcal R}(G)$}
Let ${\mathcal R}(G)<\infty$. {\em What is the range of  ${\mathcal R}(G)$?}   Since for a fixed $n$ there are finitely many graphs $G$ this range is a countable subset of the interval $[1/\sqrt{2},\infty)$. 

 {\em What is the maximum value of  ${\mathcal R}(G)$? Can ${\mathcal R}(G)$ be greater than 1?}

\subsection{Monotonicity and convexity of the function $F_G(t)$}
Lemma \ref{lemma32} states that the function $\Phi_G(x)$ is increasing on $I_G$.  If the circumcenter of a simplex  $\Delta_G(x)$ lies in this simplex, then its circumradius and the radius of the minimum enclosing sphere are the same, i.e.    $F_G(t)=\Phi^2_G(x)$,  $x=\sqrt{2t}$. Therefore, under this constraint $F_G(t)$ is monotonic. Our conjecture is:

\medskip 

{\em  $F_G(t)$ is a monotonic increasing function for all $t\in(1,\tau_1(G))$.}

\medskip 

\noindent Moreover, we think that 

\medskip 

{\em $F_G(t)$  is convex on the interval $(1,\tau_1(G))$.} 


\subsection{The second distance $\beta_*(G)$}
There are two interesting questions about $\beta_*(G)$:\\ {\em (1) What is the range of $\beta_*(G)$? \\ (2) Can $\beta_*(G_1)=\beta_*(G_2)$ for distinct $G_1$ and $G_2$? }

For the second question the answer is positive.  Let $\sigma$ be a collection of positive integers $n_{1},\ldots, n_{m}$ with $m>1$. We denote 
 $$
 |\sigma|:=n_{1}+\ldots + n_{m}. 
 $$
Let  $\bar K_\sigma:=\bar K_{n_1,\ldots,n_m}$, where $\bar K_{n_1,\ldots,n_m}$ is  the graph complement of the complete $m$--partite graph $K_{n_1,\ldots,n_m}$. In other words, $\bar K_\sigma$  is the disjoint union of cliques of sizes $n_1,\ldots, n_m$. 

Einhorn and Schoenberg \cite{ES} proved that 
$$
\rp(\bar K_\sigma)=|\sigma|-1. 
$$
Moreover, the converse statement is also true.  If for a graph $G$ on $n$ vertices we have $\rp(G)=n-1$, then $G$ is $\bar K_\sigma$ for some $\sigma$ with $|\sigma|=n$. 

 
  Let $\sigma_1=(1,1,1), \,  \sigma_2=(2,2)$ and $\sigma_3=(1,4)$. Then   $\beta_*(\sigma_i)=\sqrt{3}$ for $i=1,2,3$.  
  
  Another example, 
  $$
  \sigma=(1,1,1,1,1), \, (2,2,2), \, (4,4), \, (2,8), \, (1,16). 
  $$
For all these collections   $\beta_*(\sigma)=\sqrt{5/2}$. 

It is an interesting problem {\em to describe sets of collections $\sigma$ with the same $\beta_*(\sigma)$.}

\subsection{Sets of Type II}
 In Section 4 we consider join--indecomposable spherical sets of Type I and II.  Note that if we remove a point from a J--prime set of Type I, then we obtain a set of Type II. It is not clear can we use this method to obtain all sets of Type II?  In other words,
 
 \medskip
 
 {\em Is it true that any J--prime set of Type II is a subset of a set of Type I?} 
 
 \medskip
 
 \medskip
 
 Now we consider generalizations of graph representations. 

\subsection{Spherical representations with ${\mathcal R}(G)\le R_0$} 
Let $f$ be a  spherical representation of a graph $G$ on $n$ vertices  in ${\Bbb R}^d$  as a two--distance set with $a=1$ and $b>a$.  Let $R_0$ be a positive real number. 
We say that $f$ is a {\em minimal spherical representations with ${\mathcal R}(G)\le R_0$} if the image $f(G)$ lies in a sphere of radius $R\le R_0$ with the smallest $d$.  If $G\ne K_n$, then Theorem \ref{Jrep} yields the existence of such representations with $d\le n-1$. We denote the minimum dimension $d$ by $\sr(G,R_0)$. 

Note that $\sr(G,1/\sqrt{2})=\ur(G)$.  It is easy to see that for $R_0\ge 1/\sqrt{2}$ we have 
$$
\ur(G)\ge \sr(G,R_0) \ge \sr(G). 
$$

The following theorem can be proved by the same arguments as in the proof of Theorem \ref{Jrepth}. 
\begin{theorem} \label{thR0} Let $G\ne K_n$ be a graph on $n$ vertices. Let $R_0\ge 1/\sqrt{2}$. If\, ${\mathcal R}(G)\le R_0$, then 
$$
\sr(G,R_0)=n-\mu(G)-1, \, \mbox{ otherwise } \, \sr(G,R_0)=n-1.
$$ 
\end{theorem}

Since in Theorem \ref{Jrepth} we have $\ur(G)=\sr(G)$ this theorem also holds for $\sr(G,R_0)$.  Consider interesting problem: \\ {\em Find families  of graphs $G$ with $\sr(G,R_0)=\sr(G)$.}  

Another interesting question is {\em to find the minimum $R_0$ such that  $\sr(G,R_0)=\sr(G)$ for all $G$.}  In particular, {\em is it true that this equality holds for $R_0=1$?}  (See Subsection 7.1.)


\subsection{Representations of colored $E(K_n)$ as  $s$--distance sets}
First consider an equivalent definition of graph representations. Let $G=(V(G),E(G))$ be a graph on $n$ vertices. We have $E(K_n)=E(G)\cup E(\bar G)$. Then it is can be considered as a coloring of $E(K_n)$ in two colors. Hence 
$$E(K_n)=E_1\cup E_2, \; \mbox{ where } \; E_1\cap E_2=\emptyset.$$ 
Clearly, $G$ is uniquely defined by the equation $E(G)=E_1$. 

Let $L(e):=i$ if $e\in E_i$. Then  $L:E(K_n)\to \{1, 2\}$ is a coloring of $E(K_n)$.  A representation $L$  as a two--distance set is  an embedding $f$ of $V(K_n)$ into  ${\Bbb R}^d$ such that $\dist(f(u),f(v)))=a_i$ for $[uv]\in E_i$. Here $a_2\ge a_1>0$.

\medskip

This definition can be extended to any number of colors. Let $L:E(K_n)\to \{1,\ldots,s\}$ be a coloring of the set of edges of a complete graph $K_n$. Then 
$$
E(K_n)=E_1\cup\ldots\cup E_s, \; E_i:=\{e\in E(K_n): L(e)=i\}. 
$$
We say that  an embedding $f$ of the vertex set of $K_n$ into  ${\Bbb R}^d$  is a {\em Euclidean representation of a coloring $L$  in ${\Bbb R}^d$ as an $s$--distance set} if there are $s$ positive real numbers $a_1\le \ldots\le a_s$ such that $\dist(f(u),f(v)))=a_i$ if and only if $[uv]\in E_i$.

It is easy to extend the definitions of polynomials $C_G(t)$ and $M_G(t)$ for $s$--distance sets. In this case we have multivariate polynomials $C_L(t_2,\ldots,t_{s})$ and $M_L(t_2,\ldots,t_{s})$, where $a_1=1$ and  $t_i=a_i^2$\, for $i=2,\ldots,s$.  It is clear that a Euclidean representation of $L$ is spherical only if $F_L(t_2,\ldots,t_{s})$ is well defined, where 
$$
F_L(t_2,\ldots,t_{s}):=-\frac{1}{2}\frac{M_L(t_2,\ldots,t_{s})}{C_L(t_2,\ldots,t_{s})}. 
$$

We think that the Einhorn--Schoenberg theorem and several results from this paper can be generalized for representations of colorings $L$ as $s$--distance sets. 


\subsection{Contact graph representations of $G$}  
The famous circle packing theorem (also known as the Koebe--Andreev--Thurston theorem) states that for every connected simple planar graph $G$ there is a circle packing in the plane whose contact graph is isomorphic to $G$. 
Now consider representations of a graph $G$ as the contact graph of a packing of congruent spheres in ${\mathbb R}^d$. Equivalently, the contact graph can be defined in the following way. 

Let $X$ be a finite subset of ${\Bbb R}^{d}$. Denote
$$\psi(X):=\min\limits_{x,y\in X}{\{\dist(x,y)\}}, \mbox{ where } x\ne y.$$
The {\it contact graph} $\cg(X)$  is a graph with vertices in $X$ and edges $(x,y), \, x,y\in X$, such that $\dist(x,y)=\psi(X)$. In other words, $\cg(X)$ is the contact graph of a s packing of spheres of diameter $\psi(X)$ with centers in $X$. 

Let a graph $G=(V,E)$ on $n$ vertices have at least one edge. Let $f$ be a Euclidean representation of  vertices of $G$  in ${\Bbb R}^d$. We say that $f$ with minimum $d$ is a {\em minimal Euclidean contact graph representation} if $G$ is isomorphic to $\cg(X)$, where $X=f(V)$.  If $X$ lies on a sphere then we call $f$ a  {\em minimal spherical contact graph representation}. 

There are several combinatorial properties of contact graphs, see the survey paper \cite{bezdek2013}. For instance, the degree of any vertex of $\cg(X)$, $X \subset {\Bbb R}^d$, is not exceed the kissing number $k_d$. For spherical contact graph representations in ${\mathbb S}^2$ this degree is not greater than five. Using this and other properties of $\cg(X)$ we enumerated spherical irreducible contact graphs for $n\le 11$ \cite{mus_tar14, mus_tar15}. 

\medskip

It is an interesting problem to {\em find minimal dimensions  of Euclidean and spherical contact graph representations of graphs $G$.}
 
 \medskip

\medskip

\noindent{\bf Acknowledgment.}  I am very grateful to the reviewers of this paper for their great help in improving the text and helpful comments.

\medskip

\medskip

\medskip

\medskip

 O. R. Musin, School of Mathematical and Statistical Sciences, University of Texas Rio Grande Valley,  One West University Boulevard, Brownsville, TX, 78520.

 {\it E-mail address:} oleg.musin@utrgv.edu

\end{document}